\documentclass[fleqn,final]{zzz}
\usepackage{mathrsfs}
\usepackage{color}
\usepackage{tikz}
\usepackage{float}
\usepackage{graphicx}
\usepackage{rotating}
\usepackage[pdftex]{hyperref}
\usepackage{enumerate}

\hypersetup{
   colorlinks = true,
   urlcolor = blue,
   linkcolor = red,
   citecolor = green
}

\def\cprime{$'$}

\newtheorem{thm}[lemma]{Theorem}

\makeatletter
\def\eqnarray{\stepcounter{equation}\let\@currentlabel=\theequation
\global\@eqnswtrue
\tabskip\@centering\let\\=\@eqncr
$$\halign to \displaywidth\bgroup\hfil\global\@eqcnt\z@
  $\displaystyle\tabskip\z@{##}$&\global\@eqcnt\@ne
  \hfil$\displaystyle{{}##{}}$\hfil
  &\global\@eqcnt\tw@ $\displaystyle{##}$\hfil
  \tabskip\@centering&\llap{##}\tabskip\z@\cr}

\def\endeqnarray{\@@eqncr\egroup
      \global\advance\c@equation\m@ne$$\global\@ignoretrue}

\def\@yeqncr{\@ifnextchar [{\@xeqncr}{\@xeqncr[5pt]}}
\makeatother

\newcommand{\expe}{{\mathrm e}}
\newcommand{\dd}{{\mathrm d}}

\newcommand{\R}{\mathbb{R}}
\newcommand{\C}{\mathbb{C}}
\newcommand{\Z}{\mathbb{Z}}
\newcommand{\N}{\mathbb{N}}

\renewcommand{\r}{\mathbf{r}}

\DeclareMathOperator{\sn}{sn}
\DeclareMathOperator{\cn}{cn}
\DeclareMathOperator{\dn}{dn}
\DeclareMathOperator{\ssc}{sc}

\DeclareMathOperator{\cs}{cs}

\DeclareMathOperator{\ns}{ns}
\DeclareMathOperator{\nc}{nc}
\DeclareMathOperator{\dc}{dc}

\DeclareMathOperator{\sech}{sech}
\DeclareMathOperator{\ds}{ds}
\numberwithin{equation}{section}
\renewcommand{\P}{{\sf P}}
\renewcommand{\r}{\mathbf{r}}
\renewcommand{\u}{\mathbf{u}}
\DeclareMathOperator{\sd}{sd}
\DeclareMathOperator{\GG}{G}
\DeclareMathOperator{\HH}{H}

\numberwithin{equation}{section}
\numberwithin{corollary}{section}
\numberwithin{remark}{section}
\numberwithin{theorem}{section}
\numberwithin{lemma}{section}

\newcommand\moro[1]{{\textcolor{blue}{#1}}}

\begin{document}

\renewcommand{\PaperNumber}{***}

\FirstPageHeading

\ArticleName{Internal and external harmonics in bi-cyclide coordinates}

\ShortArticleName{Internal and external harmonics in bi-cyclide coordinates}

\Author{
Brandon Alexander\,$^\ast$,
Howard S. Cohl\,$^\dag
\!\!\ $
and Hans Volkmer\,$^\S\!\!\ $}

\AuthorNameForHeading{B. Alexander, H.~S.~Cohl, H.~Volkmer}

\Address{$^\ast$ Department of Mathematics,
University of Maryland,
College Park, MD 20742 USA
} 
\EmailD{bralex1@umd.edu} 

\Address{$^\dag$~Applied and Computational Mathematics Division,
National Institute of Standards and Technology,
Gaithersburg, MD 20899-8910, USA
} 
\EmailD{howard.cohl@nist.gov}
\URLaddressD{
\href{http://www.nist.gov/itl/math/msg/howard-s-cohl.cfm}
{http://www.nist.gov/itl/math/msg/howard-s-cohl.cfm}
}

\Address{$^\S$ Department of Mathematical Sciences,
University of Wisconsin-Milwaukee,
Milwaukee, WI 53201-0413, USA
} 
\EmailD{volkmer@uwm.edu} 

\ArticleDates{Received ?? 2022 in final form ????; Published online ????}

\Abstract{
The Laplace equation in three dimensional Euclidean space is $R$-separable in bi-cyclide coordinates
leading to harmonic functions expressed in terms of Lam\'e-Wangerin functions called
internal and external bi-cyclide harmonics.
An expansion for the fundamental solution of Laplace's equation in
products of internal and external bi-cyclide harmonics is derived.
In limiting cases this expansion reduces to known expansion in bi-spherical and prolate spheroidal coordinates.
}


\Classification{35A08; 35J05; 33C05; 33C10; 33C15; 33C20; 33C45; 33C47; 33C55; 33C75}

\section{Introduction}

The Laplace equation $\Delta u=0$ is separable in various coordinate systems in three-dimensional Euclidean space
among them the bi-cyclide coordinate system.
One of the most important tasks is to find the expansion of the reciprocal distance of two points in
a series of harmonic functions that are obtained by the method of separation of variables applied to $\Delta u=0$ in  a given coordinate system.
Such expansions are known for several coordinate systems but not for all of them.
In two previous papers \cite{BiCohlVolkmerA}, \cite{BiCohlVolkmerB}, the expansion of the reciprocal distance was given in flat-ring coordinates for the first time. In this paper we derive the expansion of the reciprocal distance in terms of
harmonic functions separated in bi-cyclide coordinates.

Bi-cyclide coordinates were originally introduced by Wangerin \cite{Wangerin1875}. They can also be found in
Miller \cite[p.~211]{Miller} and Moon and Spencer \cite[p.~124]{MoonSpencer} (the connection between these forms of bi-cyclide coordinates is made in the appendix).
In these references the ordinary differential equations obtained by applying the method of separation of variables
to $\Delta u=0$ in bi-cyclide coordinates are given. However, formulas for the internal and external harmonics as well
as the corresponding expansion
of the reciprocal distance in these harmonic functions are missing.
It is the purpose of this paper to supply these missing results.

In Section \ref{coord} we define bi-cyclide coordinates in the form given by Miller
and carry out the process of separation of variables to the Laplace equation. The form of the bi-cyclide coordinates
used by Miller has the advantage that two of the separated ordinary differential equation appear
in the standard form of the Lam\'e equation. This is not the case if the coordinates are given as in
Wangerin or Moon and Spencer.
In Section \ref{LW} we review Lam\'e-Wangerin functions that appear in the definitions
of internal and external bi-cyclide harmonics. Lam\'e-Wangerin functions are particular solutions of
Lam\'e's differential equation that have recessive behavior at two neighboring regular singularities.
Moreover, an estimate for Lam\'e-Wangerin functions is given that is needed to prove convergence of
various series expansions.
In Section \ref{H1} internal and external bi-cyclide harmonics are introduced
and their main properties are established.
In Section \ref{A1} various results involving internal and external bi-cyclide harmonics are
proved. These results include the solution of a Dirichlet problem and an integral representation
of external harmonics in terms of internal harmonics. Finally, as the main result of this paper, the expansion
of the reciprocal distance of two points in a series of internal and external bi-cyclide
harmonics is given. As corollaries we find an addition theorem and integral relations for Lam\'e-Wangerin functions.
In Section \ref{H2} we introduce a second kind of internal and external bi-cyclide harmonics.
In contrast to corresponding results in flat-ring coordinates, these internal and external bi-cyclide harmonics
of the second kind can be reduced to the ones of the first kind by a Kelvin transformation.
In the final two Sections \ref{L0} and \ref{L1} we show that limiting cases of
bi-cyclide coordinates include bi-spherical and prolate spheroidal coordinates. We
connect the expansion of the reciprocal distance in bi-cyclide coordinates to the known expansions in
bi-spherical and prolate spheroidal coordinates.

\section{Bi-cyclide coordinates}\label{coord}

Miller \cite[p.~211, (6.28)]{Miller} introduces bi-cyclide coordinates $\alpha,\beta,\phi$ in $\R^3$ by
\begin{equation}\label{bi-cyclide}
 x=R\cos\phi,\quad y=R\sin\phi,\quad z= ikR\sn(\alpha,k)\sn(\beta,k),
\end{equation}
where
\[ \frac1R=\frac{i}{k'}\left(\dn(\alpha,k)\dn(\beta,k)-k\cn(\alpha,k)\cn(\beta,k)\right).\]
Note that we corrected a typo in the definition of $R$.
These coordinates depend on a given modulus $k\in(0,1)$, and involve the Jacobian elliptic functions $\cn, \sn, \dn$
\cite[Chapter~22]{NIST:DLMF}.
We also use the complementary modulus $k'=\sqrt{1-k^2}$ and the complete elliptic integrals
of the first kind
$K=K(k)$ and $K'=K'(k)=K(k')$.
The complex coordinates $\alpha$ and $\beta$ vary in the segments $\alpha\in(iK',2K+iK')$, $\beta\in(2K-iK',2K+iK')$ and $\phi\in(-\pi,\pi]$.

Bi-cyclide coordinates can also be seen as a coordinate system in the $(R,z)$-plane, where $R=(x^2+y^2)^{1/2}$ denotes the distance of a point $(x,y,z)$ in $\R^3$ to the $z$-axis. Three dimensional bi-cyclide coordinates are then obtained by
adding the rotation angle $\phi$ about the $z$-axis.

We prefer a real version of bi-cyclide coordinates. Setting $\alpha=s+K+iK'$, $\beta=2K+it$ with $s\in(-K,K)$, $t\in (-K',K')$, we obtain
\begin{eqnarray}
&&\hspace{-9.3cm} R=\frac{\cn(s,k)\cn(t,k')}{1-\sn(s,k)\dn(t,k')},\label{RR}\\
&&\hspace{-9.3cm} z=\frac{\dn(s,k)\sn(t,k')}{1-\sn(s,k)\dn(t,k')}.\label{zz}
\end{eqnarray}
In the derivation of \eqref{RR}, \eqref{zz}, standard identities for Jacobian elliptic functions are used
\cite[\S 22.4, \S22.6]{NIST:DLMF}.

The mapping $(s,t)\in(-K,K)\times (-K',K')\mapsto (R,z)\in(0,\infty)\times\R$ is bijective, (real) analytic and its inverse is also analytic.
We omit the proofs of these statements. They are similar to proofs of corresponding statements for flat-ring
coordinates \cite{BiCohlVolkmerA2}. In fact, planar bi-cyclide and planar flat-ring  coordinates are closely related as
can be seen as follows.
Letting $s=\sigma-K$, $\sigma\in(0,2K)$ and $t=K'-\tau$, $\tau\in(0,K')$, we obtain
\[ z=T^{-1},\quad R=(x^2+y^2)^{1/2}=-\frac{ik}{T}\sn(\sigma,k)\sn(i\tau,k), \]
where
\[
T=\frac1{k'}\dn(\sigma,k)\dn(i\tau,k)+\frac{k}{k'}\cn(\sigma,k)\cn(i\tau,k) .
\]
Thus $\sigma,\tau$ are exactly the planar flat-ring coordinates treated in \cite[\S 2.2]{BiCohlVolkmerA2}
except that $R,z$ are interchanged.
Therefore, we can say that planar flat-ring and planar bi-cyclide coordinates are the same (in the first quadrant) but
their three-dimensional versions become different because we rotate about different coordinate axes.

We extend planar bi-cyclide coordinates to the $z$-axis as follows.
We note that the denominator on the right-hand sides of \eqref{RR}, \eqref{zz} is positive on the rectangle $(s,t)\in[-K,K]\times[-K',K']$
with the exception of the point $s=K, t=0$. Therefore, $R,z$ are continuous functions on this rectangle with
the point $(K,0)$ removed. The points on the boundary of the rectangle are mapped to $R=0$.
As we go around the boundary of this rectangle  in a clockwise direction as shown in Figure \ref{fig1}, $z$ transverses the
$z$-axis from $-\infty$ to $+\infty$. The segments $\gamma_j$ are mapped to $\Gamma_j$ for each
$j=1,2,3,4,5$ as shown in Figure \ref{fig2} using the notation
\begin{equation}\label{b}
 b=\frac{1-k}{k'}=\frac{k'}{1+k}\in(0,1) .
\end{equation}

\begin{figure}
\setlength{\unitlength}{2cm}
\thicklines
\begin{center}
\begin{picture}(3.5,2.5)(0,-1.3)
\put(0,-1){\line(1,0){3}}
\put(0,1){\line(1,0){3}}
\put(0,-1){\line(0,1){2}}
\put(3,-1){\line(0,1){2}}
\put(0,1){\circle*{0.1}}
\put(0,-1){\circle*{0.1}}
\put(3,1){\circle*{0.1}}
\put(3,0){\circle*{0.1}}
\put(3,-1){\circle*{0.1}}
\put(0,1){\vector(1,0){1.5}}
\put(3,-1){\vector(-1,0){1.5}}
\put(0,-1){\vector(0,1){1}}
\put(3,1){\vector(0,-1){0.5}}
\put(3,0){\vector(0,-1){0.5}}
\put(-0.3,1.1){$(-K,K')$}
\put(-0.3,-1.25){$(-K,-K')$}
\put(2.7,1.1){$(K,K')$}
\put(2.7,-1.25){$(K,-K')$}
\put(3.1,-0.1){$(K,0)$}
\put(3.1,-0.5){$\gamma_1$}
\put(1.5,-1.25){$\gamma_2$}
\put(-0.3,-0.1){$\gamma_3$}
\put(1.4,1.1){$\gamma_4$}
\put(3.1,0.5){$\gamma_5$}
\end{picture}
\caption{The rectangle $[-K,K]\times [-K',K']$ of coordinates $s,t$.\label{fig1}}
\end{center}
\end{figure}
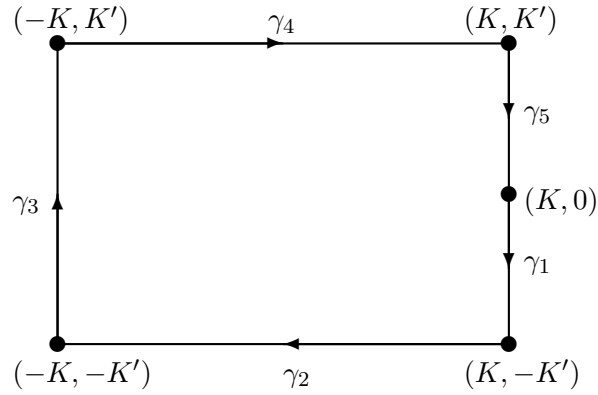

 \begin{figure}
\setlength{\unitlength}{1.8cm}
\thicklines
\begin{center}
\begin{picture}(6,1)(-3,-0.5)
\put(-3,0){\line(1,0){6}}
\put(-2,0){\circle*{0.1}}
\put(-1,0){\circle*{0.1}}
\put(1,0){\circle*{0.1}}
\put(2,0){\circle*{0.1}}
\put(-3,0){\vector(1,0){0.5}}
\put(-2,0){\vector(1,0){0.5}}
\put(-1,0){\vector(1,0){1}}
\put(1,0){\vector(1,0){0.5}}
\put(2,0){\vector(1,0){0.5}}
\put(-2.2,-0.3){$-\frac1b$}
\put(-1.2,-0.3){$-b$}
\put(0.95,-0.3){$b$}
\put(1.95,-0.3){$\frac1b$}
\put(-2.7,0.2){$\Gamma_1$}
\put(-1.7,0.2){$\Gamma_2$}
\put(-0.2,0.2){$\Gamma_3$}
\put(1.3,0.2){$\Gamma_4$}
\put(2.3,0.2){$\Gamma_5$}
\end{picture}
\caption{Bi-cyclide coordinates on the $z$-axis.\label{fig2}}
\end{center}
\end{figure}
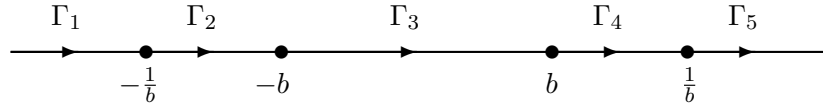

For $s_0\in(-K,K)$ and $t_0\in(-K',K')$ we introduce the polynomials
\begin{eqnarray}
&&\hspace{-1.8cm}P_1(x,y,z)= \sn^2(t_0,k')(R^2+z^2+1)^2-k^2\sd(t_0,k')^2(R^2+z^2-1)^2-4z^2,\label{P1}\\
&&\hspace{-1.8cm}P_2(x,y,z)= \sn^2(s_0,k) (R^2+z^2+1)^2-(R^2+z^2-1)^2-4k'^2\sd^2(s_0,k)z^2\label{P2},
\end{eqnarray}
where we used Glaisher's notation for the Jacobi elliptic functions \cite[(22.2.10)]{NIST:DLMF}.
If $s,t$ are bi-cyclide coordinates of $(x,y,z)$ then
\begin{equation}\label{sums}
 R^2+z^2+1=\frac{2}{1-\sn(s,k)\dn(t,k')},\quad R^2+z^2-1=\frac{2\sn(s,k)\dn(t,k')}{1-\sn(s,k)\dn(t,k')},
 \end{equation}
so a computation gives
\begin{eqnarray}
&&\hspace{-3.2cm}P_1(x,y,z)= \frac{4(\sn^2(t_0,k')-\sn^2(t,k'))(\dn^2(t_0,k')-k^2\sn^2(s,k))}{\dn^2(t_0,k')(1-\sn(s,k)\dn(t,k'))^2},
\label{P1a}\\
&&\hspace{-3.2cm}P_2(x,y,z)= \frac{4(\sn^2(s_0,k)-\sn^2(s,k))(\dn^2(t,k')-k^2\sn^2(s_0,k))}{\dn^2(s_0,k)(1-\sn(s,k)\dn(t,k'))^2}.
\label{P2a}
\end{eqnarray}
Therefore, $P_1(x,y,z)=0$ if and only if $t=t_0$ or $t=-t_0$, and $P_2(x,y,z)=0$ if and only if
$s=s_0$ or $s=-s_0$.

Figure \ref{fig3} depicts coordinate lines of planar bi-cyclide coordinates.
The coordinate lines $s=s_0$ and $t=t_0$ are shown in blue and red, respectively.
The coordinate line $t=0$ is the positive $R$-axis, and the coordinate line $s=0$ is half the unit circle.
The mapping $(s,t)\to (s,-t)$ corresponds to $(R,z)\to(R,-z)$, and $(s,t)\to(-s,t)$ corresponds to
$(R,z)\to (R^2+z^2)^{-1}(R,z)$, the inversion at the unit circle. The rectangle $(s,t)\in(0,K)\times (0,K')$\
corresponds to the region $\{(R,z): R,z>0, R^2+z^2>1\}$.
Figure \ref{fig3} also shows the position of the four points $(0,\pm b)$, $(0,\pm b^{-1})$ on the $z$-axis.

\begin{figure}[ht]
\begin{center}
\includegraphics[clip=true,trim={0 17cm 0 1cm},width=20cm]{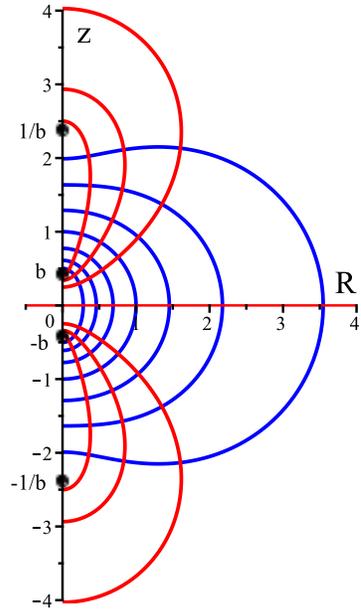}
\caption{Coordinate lines $s=\pm \frac35K,\pm \frac25K,\pm \frac15K,0$ in blue and $t=\pm \frac45K',\pm\frac35K',\frac25K',0$
in red for bi-cyclide coordinates with $k=0.7$.
\label{fig3}}
\end{center}
\end{figure}

The Laplace equation $\Delta u=0$ has ${\mathcal R}$-separated solutions
\[ u(x,y,z)= R^{-1/2} u_1(\alpha)u_2(\beta) \expe^{im\phi},\quad R=(x^2+y^2)^{1/2},
\]
where $u_1$ and $u_2$ satisfy the
ordinary differential equation
\begin{equation}\label{lame0}
\frac{d^2w}{d\zeta^2}+\left(\lambda-(m^2-\tfrac14)k^2 \sn^2(\zeta,k)\right) w =0 .
\end{equation}
This is stated in \cite[p.~211, (6.28)]{Miller} and will be confirmed in Theorem \ref{separation} below.

The Lam\'e equation is
\begin{equation}\label{lame}
 \frac{d^2w}{d\zeta^2}+(\lambda-\nu(\nu+1)k^2\sn^2(\zeta,k))w=0 .
\end{equation}
So \eqref{lame0} is the Lam\'e equation with $\nu=|m|-\frac12$.

If we write $v_1(s)=u_1(s+K+iK')$ and $v_2(t)=u_2(2K+it)$, we obtain the differential equations
\begin{eqnarray}\label{lame1}
&&\hspace{-7cm}\frac{d^2v_1}{ds^2}+(\lambda-(m^2-\tfrac14)\dc^2(s,k)) v_1=0,\\
&&\hspace{-7cm}\frac{d^2v_2}{dt^2}-(\lambda+(m^2-\tfrac14)k^2\ssc^2(t,k')) v_2=0.\label{lame2}
\end{eqnarray}
Using $\dc^2(s,k)=1+k'^2\ssc^2(s,k)$ we see that equation \eqref{lame1} is the same as \eqref{lame2} with $k$ replaced by $k'$ and $\lambda$ replaced by $\nu(\nu+1)-\lambda$. We summarize the result in the following theorem.

\begin{thm}\label{separation}
If $m\in\Z$, $\lambda\in\R$, $v_1$ solves \eqref{lame1} on $(-K,K)$ and $v_2$ solves \eqref{lame2} on $(-K',K')$. Then
\begin{equation}\label{separation1}
u(x,y,z)=R^{-1/2}v_1(s)v_2(t)\expe^{im\phi}
\end{equation}
is a harmonic function in $\R^3\setminus\{(0,0,z): z\in\R\}$.
\end{thm}
\begin{proof}
The metric coefficients of bi-cyclide coordinates are given by $h_\phi=R$ and
\begin{equation}\label{metric}
 h_s=h_t=R\,(\dc^2(s,k)+k^2\ssc^2(t,k'))^{1/2} .
 \end{equation}
In cylindrical coordinates $R,z,\phi$, the Laplace equation $\Delta u=0$ takes the form
\[ \frac{\partial^2 v}{\partial R^2}+\frac{\partial^2 v}{\partial z^2}+R^{-2}\left(\frac{\partial^2 v}{\partial \phi^2}+\frac14v\right)=0,\]
where $u=R^{-1/2} v$. Using $h_s=h_t$ this equation transforms to
\[ \frac{\partial^2 v}{\partial s^2}+\frac{\partial^2 v}{\partial t^2}+R^{-2} h_s^2\left(\frac{\partial^2v}{\partial\phi^2}+\frac14 v\right) =0.\]
We now easily confirm that $v_1(s)v_2(t) \expe^{im\phi}$ satisfies this equation.
\end{proof}

\section{Lam\'e-Wangerin functions}\label{LW}

We recall the Lam\'e-Wangerin eigenvalue problem. The Lam\'e equation \eqref{lame} has regular singular points at $\zeta=iK'$ and $\zeta=2K+iK'$ with exponents $-\nu$ and $\nu+1$ at both points. The eigenvalue problem asks for solutions of
\eqref{lame} (with $\nu\ge -\frac12$) on the segment $\zeta\in(iK',2K+iK')$ which belong to the exponent $\nu+1$
at both end points $iK'$ and $2K+iK'$.
In \cite[\S15.6]{ErdelyiHTFIII} the eigenfunctions of this eigenvalue problem are denoted by
$F_\nu^n(\zeta,k^2)$. Alternatively, in \cite{BiCohlVolkmerB} we used the notation $W_\nu^n(s,k)=F_\nu^n(s+K+iK',k^2)$.

We list the most important properties of the Lam\'e-Wangerin functions $W_\nu^n(s,k)$, $\nu\ge -\frac12$, $n\in\N_0$.
See \cite{Volkmer2018} for further details.\\[0.2cm]

\noindent (a)
The function $W_\nu^n(s,k)$ is real-valued on the interval $s\in(-K,K)$ and has exactly $n$ zeros in this open interval.
In addition, $W_\nu^n(s,k)\to 0$ as $s\to\pm K$.

\noindent (b)
For $s<K$ close to $K$ we have the expansion
\begin{equation}\label{propb}
 W_\nu^n(s,k)=\sum_{\ell=0}^\infty c_\ell (K-s)^{\nu+1+2\ell},
\end{equation}
with real coefficients $c_\ell$ and $c_0\ne 0$.

\noindent (c)
The function $W_\nu^n(s,k)$ is even/odd with $n$:
\begin{equation}\label{evenodd}
W_\nu^n(-s,k)=(-1)^n W^n_\nu(s,k).
\end{equation}

\noindent (d)
For every fixed $\nu\ge -\frac12$ and $k\in(0,1)$, the sequence of functions $\{W_\nu^n(s,k)\}_{n\in\N_0}$ forms an orthonormal basis of
the Hilbert space $L^2(-K,K)$.

\noindent (e) The function $W(s)=W_\nu^n(s)$ satisfies differential equation
\begin{equation}\label{odeW}
 \frac{d^2W}{ds^2}+\left(\Lambda_\nu^n(k)-\nu(\nu+1)\dc^2(s,k)\right) W =0,
 \end{equation}
where $\Lambda_\nu^n(k)$ denotes an eigenvalue. Some properties of these eigenvalues are given in \cite[\S 2]{BiCohlVolkmerB}.

\noindent (f)
The function $W_\nu^n(s)$ can be continued analytically to an analytic function in the strip $|\Im s|<2K'$\
with branch cuts $(-\infty,-K]$ and $[K,+\infty)$ removed.

It should be mentioned that there are no explicit formulas for the Lam\'e-Wangerin functions $W_\nu^n$ nor for the eigenvalues
$\Lambda_\nu^n$. However, efficient methods for their numerical computation are available.

In the application to bi-cyclide coordinates we use Lam\'e-Wangerin functions $W_\nu^n(s,k)$
not only for $s\in(-K,K)$ but also for complex $s$ with real part $-K$ or $K$.
This is an important difference to the application of Lam\'e-Wangerin functions to flat-ring coordinates \cite{BiCohlVolkmerB}.
In that case, $W_\nu^n(s,k)$ was used for purely imaginary~$s$.
The following two lemmas state a property of Lam\'e-Wangerin functions with complex argument that is required in the subsequent analysis.

\begin{lemma}\label{estimate1}
Let $\nu\ge 0$, $n\in\N_0$, $k\in(0,1)$.

\noindent(i)
$W_\nu^n(K+ir,k)\ne 0$ for all $r\in(0,2K')$.

\noindent(ii)
If $0<r_1<r_2<2K'$ then
\[ 0<\frac{W_\nu^n(K+ir_1,k)}{W_\nu^n(K+ir_2,k)}\le 2 \expe^{-\omega(n+\nu+1)(r_2-r_1)},
\quad\text{where $\omega:=\frac{\pi}{2K}$}.
\]
\end{lemma}
\begin{proof}
The function $W_\nu^n(K+ir,k)$, $r\in(0,2K')$, is usually not real-valued.
However, it follows from \eqref{propb} that we can write
$W_\nu^n(K+ir,k)= C w(r)$ with a suitable complex constant $C$
such that $w$ is real-valued and has the expansion
\begin{equation}\label{expw}
 w(r)=\sum_{\ell=0}^\infty d_\ell r^{\nu+1+2\ell} \quad\text{with $d_\ell\in\R$, $d_0=1$},
\end{equation}
for small $r>0$.
We will replace $W_\nu^n(K+ir,k)$ by $w(r)$ in the proof.
Now \eqref{odeW} gives
\begin{equation}\label{odew}
 w''=q(r)w,\quad q(r)=\Lambda_\nu^n(k)+\nu(\nu+1)\cs^2(r,k') .
\end{equation}
Using $\nu\ge 0$ and \cite[Lemma 2.3]{BiCohlVolkmerB}, we find
\[ q(r)\ge \Lambda_\nu^n(k)\ge \gamma^2,\quad \gamma:=\omega(n+\nu+1)>0. \]
It follows from \eqref{expw} that $w(r)>0$ and $w'(r)>0$ for small $r>0$ so \eqref{odew} and $q(r)>0$ imply $w(r)>0$ and $w'(r)>0$ for all $r\in(0,2K')$.
Now $u=w'/w$ satisfies the Riccati equation $u'+u^2=q(r)$, so by comparison with the equation $v'+v^2=\gamma^2$,
\[ u(r)\ge \gamma\tanh(\gamma(r-r_1))\quad\text{for $r\ge r_1$}.\]
Integrating from $r=r_1$ to $r=r_2$ gives
\[ \ln \frac{w(r_2)}{w(r_1)}\ge \ln\cosh(\gamma(r_2-r_1))\ge \ln\left(\tfrac12 \expe^{\gamma(r_2-r_1)}\right), \]
as desired.
\end{proof}

The proof of Lemma \ref{estimate1} does not work for negative $\nu$. When working with bi-cyclide coordinates we only
use $\nu$ of the form $\nu=m-\frac12$ with $m\in\N_0$, so it is sufficient to treat the case $\nu=-\frac12$ in the following lemma.

\begin{lemma}\label{estimate2}
Let $k\in(0,1)$.

\noindent (i)
For each $n\in\N_0$, $r\in(0,2K')$, $W_{-1/2}^n(K+ir,k)\ne 0$.

\noindent (ii)
Let $0<r_1<r_2<2K'$. Then there exist positive constants $N$ such that
\[  0<\frac{W_{-1/2}^n(K+ir_1,k)}{W_{-1/2}^n(K+ir_2,k)}\le  2\expe^{-\tfrac{\sqrt3}{2}\omega(n+\frac12)(r_2-r_1)}
\quad\text{for $n\ge N$}.
\]
\end{lemma}
\begin{proof}
As in the proof of Lemma \ref{estimate1} we replace $W_{-1/2}^n(K+ir,k)$ by the function $w_n(r)$
which satisfies differential equation \eqref{odew} with $\nu=-\frac12$ and admits the expansion \eqref{expw} with $\nu=-\frac12$. We abbreviate $\lambda_n=\Lambda_{-1/2}^n(k)$.

\noindent (i) We set $w_n(r)=\sn^{1/2}(r,k')u_n(r)$, $r\in(0,2K')$. Then equation \eqref{odew} transforms to
\begin{equation}\label{odeu}
 u_n''+\ds(r,k')\cn(r,k')u_n'- p(r)u_n=0,
\end{equation}
where
\[ p(r)=\lambda_n-\tfrac34 k'^2\sn^2(r,k')+\tfrac14 k'^2+\tfrac12 .\]
By \cite[Lemma 2.3]{BiCohlVolkmerB}, $\lambda_n\ge \frac12\omega^2-\frac14$.
Since $\omega^2>k'>k'^2$, this gives
\[ p(r)\ge \tfrac12 k'^2-\tfrac14-\tfrac34 k'^2+\tfrac14k'^2+\tfrac12=\tfrac14>0\quad\text{for $r\in(0,2K')$}.\]
Now \eqref{odeu} yields $u_n(r)=1+cr^2+\dots$ for $r$ close to $0$ with $c=\frac1{16}(2 + 4\lambda_n +k'^2)>0$,
so $u_n(r)>0$, $u_n'(r)>0$ for small positive $r$.
Since $p(r)>0$, equation \eqref{odeu}
shows that $u_n(r)>0$, $u_n'(r)>0$ for all $r\in(0,2K')$. Therefore, $w_n(r)>0$ for $r\in(0,2K')$ and
$w_n'(r)>0$ for $r\in(0,K')$. Note that we cannot show that $w_n'(r)>0$ for all $r\in(0,2K')$ because
$w_n(r)\to0$ as $r\to 2K'$  (the regular singularity $r=2K'$ of \eqref{odew} has two negative exponents $-\frac12,-\frac12$.)
This proves (i).

\noindent (ii)
We are using equation \eqref{odew}. Let $N$ be so large that $\lambda_n>1$ for $n\ge N$.
For $n\ge N$, we consider the interval
\[ I_n:=\left[\lambda_n^{-1/2} K',2K'-\lambda_n^{-1/2}K'\right] .\]
Since $\sn(r,k')$ is a concave function of $r\in[0,2K']$, we have $\sn(r,k')\ge \frac{r}{K'}$ for $r\in[0,K']$.
Therefore,
\[ q(r)\ge  \lambda_n+\tfrac14-\frac{K'^2}{4r^2}\quad\text{for $r\in(0,K']$}.  \]
This implies that
\begin{equation}\label{estq}
 q(r)\ge \tfrac34\lambda_n+\tfrac14>0\quad\text{for $r\in I_n$} .
\end{equation}
In (i) we proved that $w_n$ and $w_n'$ are positive on the interval $(0,K']$, so \eqref{odew} and \eqref{estq}
show that $w_n$ and $w_n'$ are positive on $(0,K')\cup I_n$.
Now choose $N$ so large that $r_1,r_2\in I_n$ for $n\ge N$. Arguing as in the proof of Lemma \ref{estimate2}, we obtain from \eqref{odew} and \eqref{estq} that
\[ 0<\frac{w_n(r_1)}{w_n(r_2)}\le 2\exp\left(-\tfrac{\sqrt3}{2}(\lambda_n+\tfrac14)^{1/2}(r_2-r_1)\right) .\]
This together with \cite[Lemma 2.3]{BiCohlVolkmerB} yields the desired estimate.
\end{proof}

\section{Harmonics of the first kind}\label{H1}

The coordinate surface $t=0$ is the plane $z=0$. If $t_0\in(0,K')$
then the closed coordinate surface $t=t_0$ is the part of the cyclidic surface $P_1(x,y,z)=0$ with $P_1$ defined in \eqref{P1}
which lies in the half-space $z>0$.
Similarly, if $t_0\in(-K',0)$ then the coordinate surface $t=t_0$ is given
by the part of the  surface $P_1(x,y,z)=0$ which lies in the half-space $z<0$. These surfaces are shown in red in
Figure~\ref{fig4}.

Let $t_0\in(0,K')$. Then the bounded domain $D_1$ interior to the surface $t=t_0$ is given by $t\in(t_0,K']$ in bi-cyclide coordinates and by
\[ D_1=\{(x,y,z): P_1(x,y,z)<0, z>0\} \]
in Cartesian coordinates. Its boundary is the coordinate surface $t=t_0$.

\begin{figure}[ht]
\centering
\includegraphics[clip=true,trim={0.8cm 5cm 0.6cm 5cm},width=12cm]{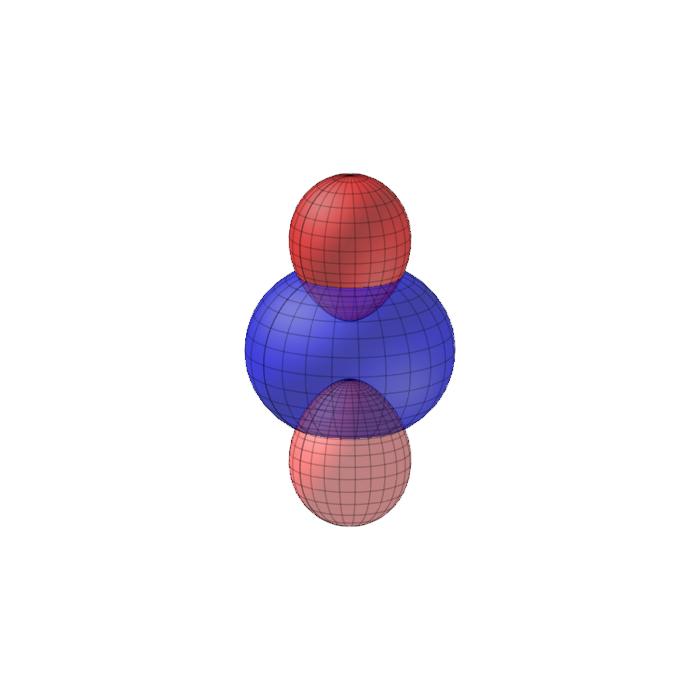}
\caption{Coordinate surfaces $s=0.2K$ in blue and $t=\pm 0.5K'$ in red of system \eqref{bi-cyclide} with $k=0.5$.\label{fig4}}
\end{figure}

\begin{figure}[ht]
\centering
   \includegraphics[clip=true,trim={0.5cm 0.5cm 0.5cm 0.5cm},width=15cm]{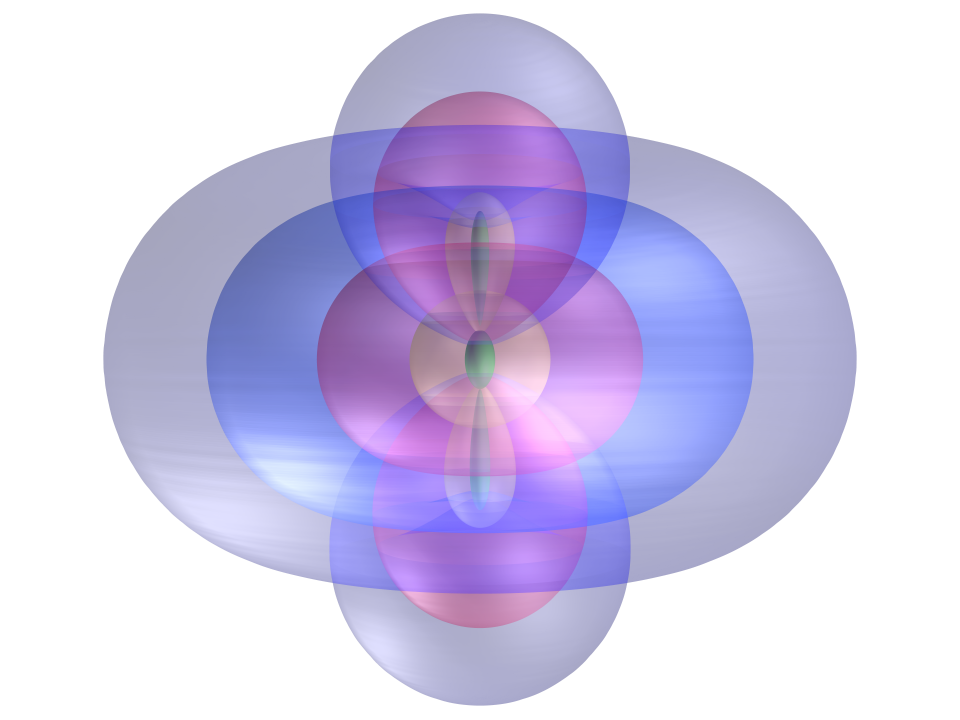}
   \caption{{For $k=0.7$ this figure depicts a three-dimensional visualization of rotationally-invariant bi-cyclides for $t\in\{\pm 0.29K',\pm 0.38K',\pm 0.70K',\pm 0.90K'\}$ (respectively green, yellow, red, blue) and orthogonal bi-concave disk cyclides $s\in\{-0.6K,0,0.48K,0.66K,0.74K\}$} (respectively green, yellow, red, blue and dark blue). Note that biconcave disk at $s=0$ (rendered in yellow) corresponds to the unit sphere. \label{flatring3d2}}
\end{figure}

We now introduce harmonic functions $u(x,y,z)$ of the separated form \eqref{separation1}
which are harmonic in the union of all $D_1$ with $t_0\in(0,K')$. In particular, these functions must be harmonic
on the positive $z$-axis.
For $m\in\Z$, $n\in\N_0$, we define internal bi-cyclide harmonics of the first kind by
\begin{equation}\label{definternal1}
 \GG_{m,n}(x,y,z)=R^{-1/2}W_{|m|-\frac12}^n(s,k)W_{|m|-\frac12}^n(it-K-iK',k)\expe^{im\phi} .
 \end{equation}

\begin{thm}\label{internal1}
The internal bi-cyclide harmonic $\GG_{m,n}(x,y,z)$ is harmonic on all of $\R^3$ with the exception of the segment
$\{(0,0,z): -b^{-1}\le z\le -b\}$, where $b$ is given by \eqref{b}.
\end{thm}
\begin{proof}
Using \eqref{odeW} we see that $v_1(s)=W_{|m|-\frac12}(s,k)$,
 $v_2(t)=W_{|m|-\frac12}^n(it-K-iK',k)$ satisfy \eqref{lame1}, \eqref{lame2} with $\lambda=\Lambda_{|m|-\frac12}^n(k)$
in both equations.
It follows from Theorem~\ref{separation} that $\GG_{m,n}(x,y,z)$ is harmonic on all of $\R^3$ minus the $z$-axis.
Using \eqref{RR}, \eqref{propb} and \eqref{evenodd}, we see that
the function
\[(s,t)\mapsto R^{-1/2}W_{|m|-\frac12}^n(s,k)W_{|m|-\frac12}^n(it-K-iK',k)\]
is locally bounded at every point
on the boundary of the rectangle $(s,t)\in(-K,K)\times(-K',K')$ with the exception of the closed segment $\gamma_2$ (defined in Figure \ref{fig1}) and the point $(K,0)$.
Since the map $(R,z)\mapsto (s,t)$ is continuous, we obtain that $\GG_{m,n}(x,y,z)$ is locally bounded
at every point of the $z$-axis with the exception of the closed segment $\Gamma_2$ (defined in Figure \ref{fig2}).
Note that we cannot claim that
$\GG_{m,n}(x,y,z)$ is locally bounded at the points of the closed segment $\Gamma_2$
because the function $W^n_\nu$ is a solution of \eqref{odeW} which belongs to the exponent $\nu+1$ at the regular singular points $-K$ and $K$ but possibly not at $K-2iK'$ (actually, it cannot belong to the exponent $\nu+1$ there.)
The local boundedness of $\GG_{m,n}$ at a point on the $z$-axis implies that the function $\GG_{m,n}$ can be continued to
an harmonic function in a neighborhood of this point according to the following lemma. This completes the proof.
\end{proof}

\begin{lemma}
Consider the ball
\[ B_r=\{(x,y,z): x^2+y^2+z^2<r^2\}, \]
and a bounded continuous function
\[u:B_r^\ast:=\overline B_r\setminus\{(0,0,z): z\in\R\}\to\R\]
such that $u$ is harmonic on
$B_r\setminus\{(0,0,z): z\in\R\}$.
Then $u$ has a harmonic extension to $B_r$.
\end{lemma}
\begin{proof}
Using the Poisson integral \cite[p.~241]{Kellogg}
we solve the Dirichlet problem on $B_r$ with boundary values $u$. We obtain a solution $U$ which is harmonic on $B_r$,
continuous on $\overline B_r\setminus \{(0,0,r),(0,0,-r)\}$ and agrees with $u$ on $\partial B_r\setminus \{(0,0,r),(0,0,-r)\}$
\cite[p.~243, Remark]{Kellogg}.
Define a function $v$ by
\[ v(x,y,z)=-\ln\frac{\sqrt{x^2+y^2}}{r} .\]
This function is harmonic on $\R^3$ minus the $z$-axis. Let $\epsilon>0$.
Consider the function
\[ w=\epsilon v-(u-U) .\]
Then $w\ge 0$ on $\partial B_r\setminus \{(0,0,r),(0,0,-r)\}$.
There is a constant $M$ such that $|u|\le M$ on $B_r^\ast$. Then also $|U|\le M$ on $B_r^\ast$, so $|u-U|\le 2M$ on $B_r^\ast$.
Choose $\delta>0$ so small that $\epsilon v\ge 2M$ if $x^2+y^2\le \delta^2$.
Then $w\ge 0$ on the boundary of the set $A:=B_r\setminus \{(x,y,z): x^2+y^2\le \delta^2\}$.
By the maximum principle for harmonic functions, $w\ge 0$ on $\overline A$.
We can choose $\delta>0$
as small as we want, so $w\ge 0$ on $B_r^\ast$. Since $\epsilon>0$ is arbitrary, we get $U-u\ge 0$ on $B_r^\ast$.
In a similar way, we get $U-u\le0$ on $B_r^\ast$.
Therefore, $u=U$ on $B_r^\ast$, so $U$ is the desired extension of $u$.
\end{proof}

Let
\[ \sigma(\r)=\|\r\|^{-2}\r \]
denote the inversion at the unit sphere in $\R^3$.
Then the corresponding Kelvin transform of a harmonic function $u(\r)$ is
\begin{equation}\label{Kelvin}
  \hat u(\r)=\|\r\|^{-1}u(\sigma(\r))
\end{equation}
and this function is also harmonic.
The inversion at the unit sphere is expressed by $s\mapsto -s$ in bi-cyclide coordinates.
It follows from \eqref{evenodd} and
\[ x^2+y^2+z^2=\frac{1+\sn(s,k)\dn(t,k')}{1-\sn(s,k)\dn(t,k')} \]
that the Kelvin transformation \eqref{Kelvin} of $\GG_{m,n}$ satisfies
\[ \widehat{\GG}_{m,n}(\r)=(-1)^n  \GG_{m,n}(\r) .\]

For $m\in\Z$, $n\in\N_0$, we define external bi-cyclide harmonics of the first kind by
\begin{equation}\label{defexternal1}
 \HH_{m,n}(x,y,z)=R^{-1/2}W_{|m|-\frac12}^n(s,k)W_{|m|-\frac12}^n(-it-K-iK',k)\expe^{im\phi} .
\end{equation}
The definition of $\HH_{m,n}$ is the same as that of $\GG_{m,n}$ except that we replaced $t$ by $-t$.
Therefore,
\begin{equation}\label{intext1}
 \HH_{m,n}(x,y,z)=\GG_{m,n}(x,y,-z).
\end{equation}
By Theorem \ref{internal1}, $\HH(x,y,z)$ is harmonic on all of $\R^3$ except the segment $\{(0,0,z): b\le z\le b^{-1}\}$.
Note that the notions ``internal'' and ``external'' refer to the surfaces $t=t_0$ with $t_0\in(0,K')$.

\section{Applications of bi-cyclide harmonics of the first kind}\label{A1}

We solve the Dirichlet problem for the region $D_1$ given by  $t\in(t_0,K']$, where $t_0\in(0,K')$.
We say that a harmonic function $u$ defined in $D_1$ attains
the boundary values $f$ on $\partial D_1$
in the weak sense if $R^{1/2}\,u$
(expressed in terms of bi-cyclide coordinates $s,t,\phi$) evaluated at $t_1\in(t_0,K')$ converges to $R^{1/2}f$
in the Hilbert space
\[ H_1=L^2((-K,K)\times (-\pi,\pi)) \]
as $t_1\to t_0$.
As in \cite[\S~5.2]{BiCohlVolkmerA2},
one can show that the solution of the Dirichlet problem is unique.

\begin{thm}\label{Dirichlet1}
Let $f$ be a function defined on the boundary $\partial D_1$ of the region $D_1$ given by $t\in(t_0,K']$ for some $t_0\in(0,K')$.
Suppose that $f$ is represented in bi-cyclide coordinates as
\[ f(\r)=R^{-1/2}g(s,\phi) ,\quad  s\in(-K,K),\quad \phi\in(-\pi,\pi],\]
such that $g\in H_1$.
For all $m\in\Z$ and $n\in\N_0$ define
\begin{eqnarray*}
&&\hspace{-2.5cm} c_{m,n}:=\frac{1}{2\pi}\int_{-\pi}^\pi
 \expe^{-im\phi}
 \int_{-K}^K g(s,\phi)W^{n}_{|m|-\frac12}(s,k)\,\dd s\,\dd\phi\\
 &&\hspace{-1.5cm}=  \frac{1}{2\pi W^{n}_{|m|-\frac12}(it_0-K-iK',k)}\int_{\partial D_1} \frac{1}{h_s(\r)} f(\r) \GG_{-m,n}(\r)\,\dd S(\r),
\end{eqnarray*}
where $h_s$ is given in \eqref{metric}.
Then the function
\begin{equation}\label{sol1}
 u(\r)=\sum_{m\in\Z}\sum_{n=0}^\infty d_{m,n}  G_{m,n}(\r),\quad d_{m,n}:=c_{m,n}\{W_{|m|-\frac12}^n(it_0-K-iK',k)\}^{-1},
\end{equation}
is harmonic in $D_1$ and it attains the boundary values $f$ on $\partial D_1$ in the weak sense.
The infinite series in \eqref{sol1} converges absolutely and uniformly in compact subsets of $D_1$.
\end{thm}
\begin{proof}
Since $\dd S(\r)=R h_s(\r)\,\dd s\,\dd\phi$, the two formulas for $c_{m,n}$ agree.
The system of functions $W^n_{|m|-\frac12}(s,k)\expe^{im\phi}$, $m\in\Z$, $n\in\N_0$, is orthogonal and complete in
the Hilbert space $H_1$ so we have the corresponding Fourier expansion
\[ g(s,\phi)\sim \sum_{m\in\Z}\sum_{n=0} c_{m,n}W_{|m|-\frac12}^n(s,\phi)\expe^{-im\phi}.\]
In particular, the sequence $\{c_{m,n}\}$ is bounded: $|c_{m,n}|\le C_1$.
We use the Weierstrass $M$-test to show uniform convergence of the series in \eqref{sol1} on the compact set $t\ge t_1>t_0$.
Using the maximum principle for harmonic functions it is sufficient to find bounds $M_{m,n}$
such that $|d_{m,n}G_{m,n}(\r)|\le M_{m,n}$ for $t=t_1$ and $\sum_{m\in\Z}\sum_{n\in\N_0} M_{m,n}<\infty$.
Using \eqref{RR} we find for $s\in(-K,K)$ and $t=t_1$,
\[ \frac1R=\frac{1-\sn(s,k)\dn(t_1,k')}{\cn(s,k)\cn(t_1,k')}\le \frac{2}{\cn(s,k)\cn(t_1,k')}\le\frac{2}{k'}\dc(s,k)\nc(t_1,k') .\]
Using \cite[Lemmas 2.4, 2.5]{BiCohlVolkmerB} and Lemmas \ref{estimate1}, \ref{estimate2},
we estimate
\[
\left|\frac{W_{|m|-\frac12}^n(it_1-K-iK',k)}{W_{|m|-\frac12}^n(it_0-K-iK',k)}
R^{-1/2}W_{|m|-\frac12}^n(s,k)\expe^{im\phi}\right|\le
C_2 p^{|m|+n} (1+|m|+n) ,\]
where the constants $C_2$ and $p\in(0,1)$ are independent of $m,n,s,\phi$.
Therefore, we can take $M_{m,n}=C_1C_2 p^{|m|+n} (1+|m|+n)$ and the proof of convergence is complete.
Hence $u(\r)$ defined by \eqref{sol1} is a harmonic function on $D_1$.
We show that  $u$ attains the boundary values $f$ on $\partial D_1$ in the weak sense
by the same method as used in the proof of \cite[Theorem 5.3]{BiCohlVolkmerA2}.
\end{proof}

Define the Wronskian $w_{m,n}$ by
\begin{equation}\label{wronskian}
w_{m,n}:=U(t)V'(t)-U'(t)V(t),
\end{equation}
where $U(t):=W^n_{|m|-\frac12}(it-K-iK',k)$, $V(t):=U(-t)$.
External harmonics admit an integral representation in terms of internal harmonics.

\begin{thm}\label{intrep1}
Let $t_0\in(0,K')$, $m\in\Z$, $n\in\N_0$, and let $\r^\ast$ be  a point outside $\overline{D}_1$, where $D_1$ is the region given by $t\in(t_0,K']$. Then
\begin{equation}\label{intrep1a}
 \HH_{m,n}(\r^\ast)=\frac{w_{m,n}}{4\pi\{W_{|m|-\frac12}^n(it_0-K-iK',k)\}^2}\int_{\partial D_1} \frac{\GG_{m,n}(\r)}{h_s(\r)\|\r-\r^\ast\|}{\mathrm d}S(\r).
\end{equation}
\end{thm}

We omit the proof of this theorem which is very similar to the proof of
\cite[Theorem 5.5]{BiCohlVolkmerA2}. It follows from \eqref{intrep1a} that $w_{m,n}\ne 0$.

We obtain the expansion of the reciprocal distance of two points in internal and external bi-cyclide harmonics
by combining Theorems  \ref{Dirichlet1} and \ref{intrep1}.

\begin{thm}\label{expansion1}
Let $\r,\r^\ast\in \R^3$ have bi-cyclide coordinates $(s,t,\phi)$ and $(s^\ast,t^\ast,\phi^\ast)$,  respectively.
If $-K'<t^\ast<t<K'$ then
\begin{equation}\label{expansion1a}
\frac{1}{\|\r-\r^\ast\|}=2\sum_{m\in\Z}\sum_{n=0}^\infty
\frac{1}{w_{m,n}} \GG_{m,n}(\r)\HH_{-m,n}(\r^\ast).
\end{equation}
\end{thm}
\begin{proof}
If $t>0$ we choose $t_0\in(0,K')$ such that $t^\ast<t_0<t$, and consider the region $D_1$ interior to the surface $t=t_0$.
Then we apply Theorem \ref{Dirichlet1}
to the function $f(\u)=\|\u-\r^\ast\|^{-1}$ which is harmonic on an open set containing the closure of $D_1$ (because $\r^\ast$ lies outside the closure of $D_1$).
Using Theorem \ref{intrep1} to evaluate the Fourier coefficients, we obtain \eqref{expansion1a}.

If $t\le 0$, we replace $\r$ and $\r^\ast$ by their reflections at the plane $z=0$. Then $t,t^\ast$ are
replaced by $-t,-t^\ast$. Now we apply the result from the first part of the proof to the reflected points
(in reversed order) and obtain again \eqref{expansion1} observing \eqref{intext1}.
\end{proof}

As a corollary we obtain the following addition formula for Lam\'e-Wangerin functions.
\begin{thm}\label{addnthm1}
Let {$m\in\N_0$}, $s,s^\ast\in(-K,K)$, $-K'<t^\ast<t<K'$.
Then
\begin{eqnarray}
&&\hspace{-2.9cm}Q_{m-\frac12}(\chi)
=
2\pi
\sum_{n=0}^\infty
\frac{1}{w_{m,n}}
W^n_{{m}-\frac12}(s,k)
W^n_{{m}-\frac12}(it-K-iK',k)\\
&&\hspace{1.5cm}\times W^n_{{m}-\frac12}(s^\ast,k)
W^n_{{m}-\frac12}(-it^\ast-K-iK',k),\nonumber
\end{eqnarray}
where $\chi:((-K,K)\times(-K',K'))^2\times(0,1)\to(1,\infty)$ is given by
\begin{eqnarray}\label{chist}
 &&\hspace{-0.9cm}\chi(s,t,s^\ast,t^\ast,k)=\nc(s,k)\nc(t,k')\nc(s^\ast,k)\nc(t^\ast,k')\\
 &&\quad -\dc(s,k)\ssc(t,k')\dc(s^\ast,k)\ssc(t^\ast,k')
 -\ssc(s,k)\dc(t,k')\ssc(s^\ast,k)\dc(t^\ast,k'),\nonumber
\end{eqnarray}
and $w_{m,n}$ is the Wronskian \eqref{wronskian}.
\end{thm}
\begin{proof}
This follows from comparison of \eqref{expansion1a} with the azimuthal
Fourier expansion \cite[(15)]{CT}
\[ \frac{1}{\|\r-\r^\ast\|}=\frac{1}{\pi\sqrt{RR^\ast}}\sum_{m=0}^\infty Q_{m-\frac12}(\chi)\expe^{im(\phi-\phi^\ast)},\]
where
\[
\chi=\frac{R^2+{R^\ast}^2+(z-z^\ast)^2}{2RR^\ast},
\]
with $R$, $R^\ast$, $z$, $z^\ast$ given
in terms of bi-cyclide coordinates $s,t$ and $s^\ast, t^\ast$
respectively. The identity \eqref{chist} can be verified by a direct computation.
\end{proof}

Theorem \ref{addnthm1} leads to an integral relation for Lam\'e-Wangerin functions.

\begin{thm}\label{intrel1}
Let $m,n\in\N_0$, $s^\ast\in(-K,K)$, $-K'<t^\ast<t<K'$.
Then
\begin{eqnarray*}
 &&\hspace{-1.0cm}w_{m,n}\int_{-K}^K Q_{m-\frac12}(\chi) W_{m-\frac12}^n(s,k)\,\dd s\\
 &&\hspace{1.0cm}=
2\pi
W^n_{{m}-\frac12}(it-K-iK',k) W^n_{{m}-\frac12}(s^\ast,k)
W^n_{{m}-\frac12}(-it^\ast-K-iK',k).
\end{eqnarray*}
\end{thm}

Using the method employed in \cite{Volkmer84}, one can show that Theorem \ref{intrel1} remains true if we replace $m-\frac12$ everywhere by $\nu$. See \cite{BiCohlVolkmerB} for more details.

\section{Harmonics of the second kind}\label{H2}

For $s_0\in(-K,K)$ the coordinate surface $s=s_0$ is a closed surface.
An instance of this surface is shown in blue in Figure \ref{fig4}.
If $s_0=0$ this surface is the unit sphere $x^2+y^2+z^2=1$.
If $s_0\in(-K,0)$ the surface $s=s_0$ is given by the part of the cyclidic surface $P_2(x,y,z)=0$
with $P_2$ defined in \eqref{P2} which lies in the unit ball
$B=\{(x,y,z): x^2+y^2+z^2<1\}$.
If $s_0\in(0,K)$ the surface $s=s_0$ is given by the part of the surface $P_2(x,y,z)=0$
which lies outside $\bar B$.
The surface $s=s_0$ encloses the bounded domain $D_2$ given by $s\in[-K,s_0)$ in
bi-cyclide coordinates.

The coordinate surfaces $s=s_0$ and $t=t_0$ are connected through the inversion
$M$ at the sphere with center $(0,0,1)$ and radius $\sqrt{2}$.
This inversion is given by
\[ M(x,y,z)=(x^2+y^2+(z-1)^2)^{-1}(2x,2y,x^2+y^2+z^2-1) .\]
If $(x,y,z)$ has bi-cyclide coordinates $(s,t,\phi)$ then \eqref{sums} gives
\[ M(x,y,z)=(u\cos\phi,u\sin\phi,v),\]
where
\[
 u=\frac{\cn(s,k)\cn(t,k')}{1-\dn(s,k)\sn(t,k')},\quad
 v=\frac{\sn(s,k)\dn(t,k')}{1-\dn(s,k)\sn(t,k')}.
\]
Therefore, the point $M(x,y,z)$ has bi-cyclide coordinates $t,s,\phi$ (with $s,t$ exchanged)
with bi-cyclide coordinates taken with respect to the complementary modulus $k'$.
This means that the coordinate surface $s=s_0$ (with respect to $k$) is mapped to the coordinate surface $t=s_0$
(with respect to $k'$).
If $t_0\in(0,K')$ then $M$ maps the domain $D_1$ given by $t>t_0$ with respect to $k$
to the exterior of the domain $D_2$ given by $s>t_0$ with respect to $k'$.

Because of this connection between the coordinate surfaces, the results on bi-cyclide harmonics of the second kind adapted to the domains $D_2$ will be very similar to the ones for bi-cyclide harmonic of the first kind.
Therefore, we will keep the following treatment of bi-cyclide harmonics of the second kind short.

We are looking for harmonic functions $u(x,y,z)$ of the ${\mathcal R}$-separated form \eqref{separation1}
which are harmonic in the union of all $D_2$ with $s_0\in(-K,K)$. This requires that $u$ must be harmonic
on the interval $\{(0,0,z): -b^{-1}<z<b^{-1}\}$ on the $z$-axis.
For $m\in\Z$, $n\in\N_0$, we define internal bi-cyclide harmonics of the second kind by
\[ \GG_{m,n}(x,y,z)=R^{-1/2}W_{|m|-\frac12}^n(-is-K'-iK,k')W_{|m|-\frac12}^n(t,k')\expe^{im\phi} .\]

\begin{thm}\label{t3}
The internal harmonic $\GG_{m,n}(x,y,z)$ is a harmonic function on all of $\R^3$ with the exception of set
$\{(0,0,z): |z|\ge b^{-1}\}$.
\end{thm}

For $m\in\Z$, $n\in\N_0$, we define external bi-cyclide harmonics of the second kind by
\[
 \HH_{m,n}(x,y,z)=R^{-1/2}W_{|m|-\frac12}^n(is-K'-iK),k')W_{|m|-\frac12}^n(t,k')\expe^{im\phi} .
\]
Then $\HH_{m,n}$ is the Kelvin transform of $\GG_{m,n}$ with respect to the unit sphere:
\[ \HH_{m,n}(x,y,z)=\widehat{\GG}_{m,n}(x,y,z) .\]
The function $\HH_{m,n}(x,y,z)$ is harmonic
on $\R^3$ with the exception of the segment $\{(0,0,z): |z|\le b\}$.

Arguing as in Section \ref{A1} we prove the expansion of the reciprocal distance of two points in
bi-cyclide harmonics of the second kind. Alternatively, employing the inversion $M$, the result can be derived directly from
Theorem \ref{expansion1}.

\begin{thm}\label{expansion2}
Let $\r,\r^\ast\in \R^3$ with bi-cyclide coordinates $(s,t,\phi)$ and $(s^\ast,t^\ast,\phi^\ast)$,  respectively.
If $-K<s<s^\ast<K$ then
\begin{equation}\label{expansion2a}
\frac{1}{\|\r-\r^\ast\|}=2\sum_{m\in\Z}\sum_{n=0}^\infty
\frac{1}{w_{m,n}} \GG_{m,n}(\r)\HH_{-m,n}(\r^\ast),
\end{equation}
where $w_{m,n}=U(s)V'(s)-U'(s)V(s)$ is the Wronskian of the functions $U(s)=W_{|m|-\frac12}^n(is-K'-iK,k')$ and
$V(s)=U(-s)$.
\end{thm}

As in Section \ref{A1}, Theorem \ref{expansion2} yields an addition theorem and integral relations for Lam\'e-Wangerin functions.
We omit these results because they can be obtained from Theorems \ref{addnthm1} and \ref{intrel1} by exchanging $k\leftrightarrow k'$, $K\leftrightarrow K'$ and $s\leftrightarrow t$.

\section{The bi-spherical limit of
bi-cyclidic coordinates, $k\to 0$}\label{L0}

\begin{figure}[ht]
\begin{center}
\includegraphics[clip=true,trim={1.3cm 0.4cm 10.9cm 1.2cm},width=5cm]{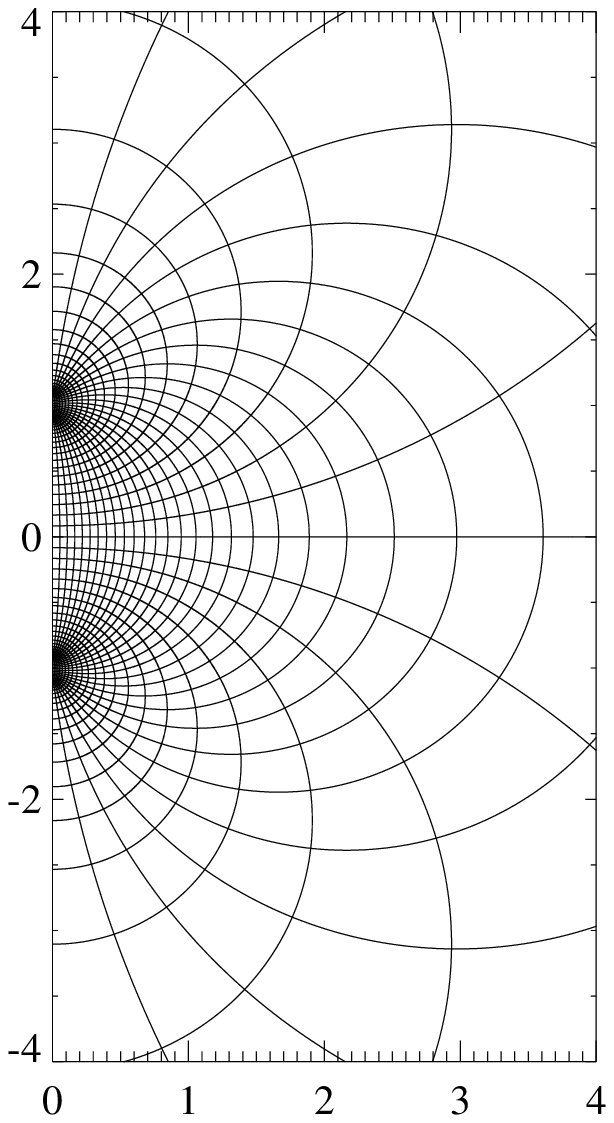}
\includegraphics[clip=true,trim={1.3cm 0.4cm 10.9cm 1.2cm},width=5cm]{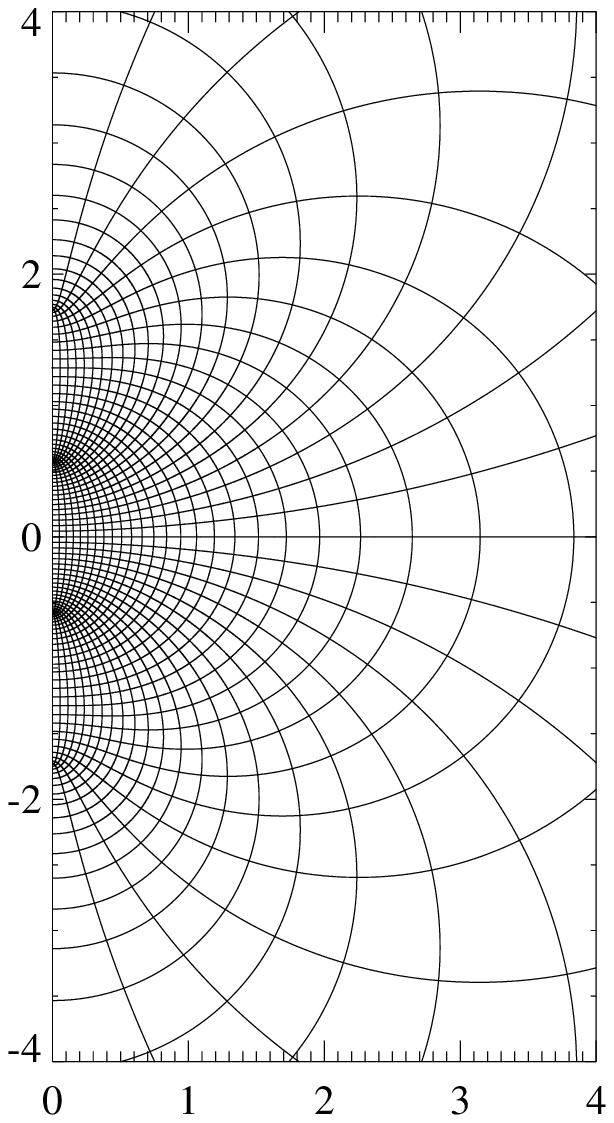}
\includegraphics[clip=true,trim={1.3cm 0.4cm 10.9cm 1.2cm},width=5cm]{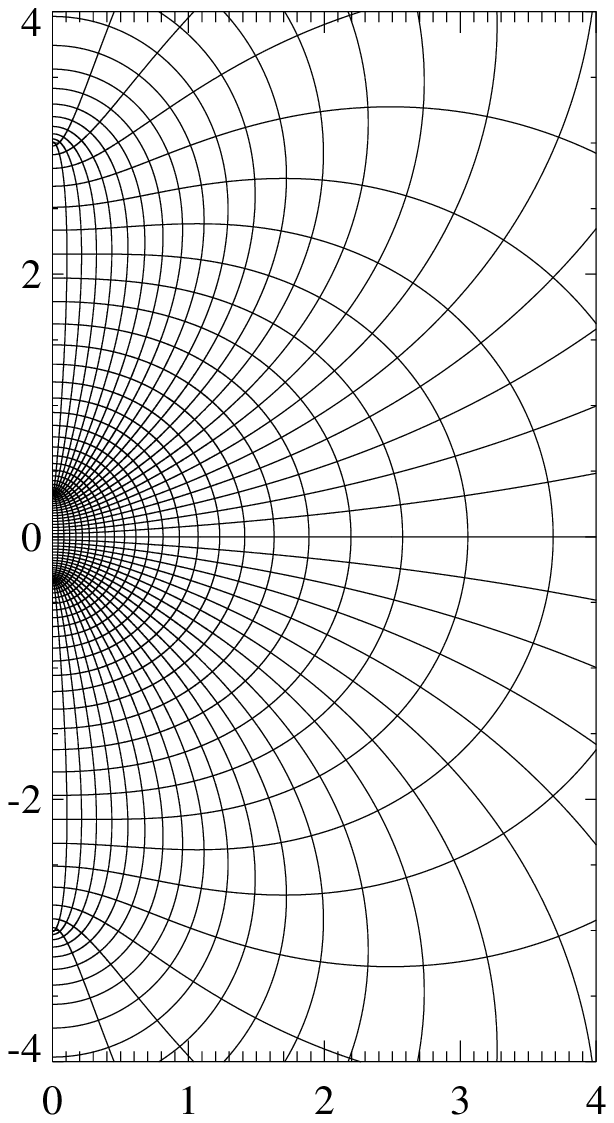}
\caption{In bi-cyclide coordinates, the figures depict coordinate lines for constant values of $s\in(-K,K)$ and $t\in(-K',K')$  with uniform spacing for $k=\frac{1}{10}, \frac{1}{2}, \frac{4}{5}$ respectively from left to right.
The abscissa represents the radial coordinate $R=(x^2+y^2)^{1/2}$ and the ordinate represents the $z$-axis. One can see that as $k$ approaches zero, the bi-cyclidic coordinate system approaches bi-spherical coordinates. Similarly, as $k$ approaches unity, the bi-cyclidic coordinate system approaches spherical coordinates.\label{4:fig1}
}
\end{center}
\end{figure}

In this section we show that bi-cyclide coordinates approach bi-spherical coordinates as $k\to 0$, and
our expansion of the reciprocal distance between two points \eqref{expansion1a} approaches term-by-term the corresponding known expansion in bi-spherical coordinates.

Bi-spherical coordinates $\theta, t, \phi$ \cite[p.~110]{MoonSpencer} are given by
\[ x=\frac{\sin\theta\cos\phi}{\cosh t-\cos\theta},\quad y=\frac{\sin\theta\sin\phi}{\cosh t-\cos\theta},\quad
z=\frac{\sinh t}{\cosh t-\cos\theta},\]
where $t\in\R$, $\theta\in(0,\pi)$, $\phi\in(-\pi,\pi]$.
According to \cite[(10.3.74)]{MorseFeshbach} we have the expansion
\begin{eqnarray*}
&&\hspace{-1.5cm} \frac{1}{\|\r-\r^\ast\|}= (\cosh t-\cos\theta)^{1/2}(\cosh t^\ast-\cos\theta^\ast)^{1/2}\\
&&\hspace{0.5cm}\times\sum_{\ell=0}^\infty \expe^{-(\ell+\frac12)(t-t^\ast)}\sum_{m=-\ell}^\ell\frac{(\ell-m)!}{(\ell+m)!}
\P_\ell^m(\cos\theta)\P_\ell^m(\cos\theta^\ast)\expe^{im(\phi-\phi^\ast)},
\end{eqnarray*}
where $\P_\ell^m$ denotes the Ferrers function of the first kind, $(\theta,t,\phi)$, $(\theta^\ast,t^\ast,\phi^\ast)$ are bi-spherical coordinates
of $\r$, $\r^\ast$, respectively, and it is assumed that $t^\ast<t$.

If a real-valued function $f(\phi)$ with period $2\pi$ is expanded in a complex Fourier series
$f(\phi) =\sum_{m\in\Z} c_m \expe^{im\phi}$,
then we must have $c_{-m}=\overline{c}_m$. In our case, the coefficients $c_m$ are real, so we have $c_{-m}=c_m$.
Therefore, we can write the expansion of $\|\r-\r^\ast\|^{-1}$ in the equivalent form
\[ \frac{1}{\|\r-\r^\ast\|}=\sum_{m\in\Z}
 \expe^{im(\phi-\phi^\ast)}\sum_{n=0}^\infty B_{m,n}(\theta,\theta^\ast,t,t^\ast),\]
where $B_{-m,n}=B_{m,n}$, and, for $m,n\in\N_0$,
\begin{eqnarray*}
&&\hspace{-2.0cm} B_{m,n}=(\cosh t-\cos\theta)^{1/2}(\cosh t^\ast-\cos\theta^\ast)^{1/2}\\
 &&\hspace{1cm}\times \expe^{-(m+n+\frac12)(t-t^\ast)} \frac{n!}{(2m+n)!}\P_{m+n}^{m}(\cos\theta)\P_{m+n}^{m}(\cos\theta^\ast).
\end{eqnarray*}

If we let $k\to 0$ in \eqref{RR}, \eqref{zz}, and observe
\cite[Tables 22.5.3, 22.5.4]{NIST:DLMF},
\[ \sn(s,k)\to \sin s, \quad \cn(s,k)\to \cos s,\quad \dn(s,k)\to 1, \]
\[ \sn(t,k')\to\tanh t,\quad \cn(t,k')\to \sech t,\quad \dn(t,k')\to \sech t,\]
we find that bi-cyclide coordinates approach bi-spherical coordinates with $s=\frac\pi2-\theta$.

Let us write the expansion \eqref{expansion1a} in the form
\[ \frac{1}{\|\r-\r^\ast\|}=\sum_{m\in\Z} \expe^{im(\phi-\phi^\ast)}\sum_{n=0}^\infty A_{m,n}(s,s^\ast,t,t^\ast,k),\]
where $A_{-m,n}=A_{m,n}$, and, for $m,n\in\N_0$,
\begin{eqnarray*}
&&\hspace{-1.5cm} A_{m,n}=\frac{2}{w_m^n}(RR^\ast)^{-1/2}
W_{m-\frac12}^n(s,k) W^n_{m-\frac12}(s^\ast,k)\\
&&\hspace{2.5cm}\times W_{m-\frac12}^n(it-K-iK',k) W_{m-\frac12}^n(-it^\ast-K-iK',k) .
\end{eqnarray*}

The following theorem states the main result of this section.

\begin{thm}\label{limitk0}
Let $m\in\Z$, $n\in\N_0$, $s,s^\ast\in(-\frac12\pi,\tfrac12\pi)$, $t,t^\ast\in\R$. Then
\[ A_{m,n}(s,s^\ast,t,t^\ast,k)\to B_{m,n}(\tfrac12\pi-s,\tfrac12\pi-s^\ast,t,t^\ast)\quad \text{as $k\to 0$} .\]
\end{thm}
\begin{proof}
It is sufficient to consider $m\ge 0$. All limits in this proof are taken as $k\to 0$.
We first note that
\begin{equation}\label{limit1}
 R^{-1/2}{R^\ast}^{-1/2}\to \left(\frac{\cos s}{\cosh t-\sin s}\right)^{-1/2}\left(\frac{\cos s^\ast}{\cosh t^\ast-\sin s^\ast }\right)^{-1/2}.
\end{equation}
Using \cite[Corollary~4.4]{BiCohlVolkmerB}, we have
\begin{eqnarray}\label{limit2}
 && \hspace{-1.0cm}W_{m-\frac12}^n(s,k) W^n_{m-\frac12}(s^\ast,k)\nonumber\\
 &&\quad\to (m+n+\tfrac12)\frac{n!}{(2m+n)!}(\cos s)^{1/2}
\P_{m+n}^n(\sin s)(\cos s^\ast)^{1/2} \P_{m+n}^m(\sin s^\ast).
\end{eqnarray}
By \cite[Theorem~4.7]{BiCohlVolkmerB}, we have
\[ \frac{W_{m-\frac12}^n(i(\sigma-K'),k)}{W_{m-\frac12}^n(-iK',k)}\to \expe^{-(m+n+\frac12)\sigma} \]
locally uniformly for $\sigma\in\C$.
Actually, it was assumed there that $|\Im \sigma|<\frac12\pi$ but the proof shows that this restriction is superfluous.
If we set $\sigma=t+iK$ and note that $K(k)\to \frac12\pi$, it follows that
\begin{eqnarray}\label{limit3}
 &&\hspace{-1.2cm}\frac{2}{w_n^m}  W_{m-\frac12}^n(it-K-iK',k) W_{m-\frac12}^n(-it^\ast-K-iK',k)\\
&& \hspace{5cm} \to \frac{\expe^{-(m+n+\frac12)(t+i\frac12\pi)}\expe^{(m+n+\frac12)(t^\ast+i\frac12\pi)}}{m+n+\frac12} .\nonumber
\end{eqnarray}
After multiplying out \eqref{limit1}, \eqref{limit2}, \eqref{limit3} and minor simplification, we obtain the desired statement.
\end{proof}

\section{
The prolate spheroidal limit of
bi-cyclidic coordinates, $k\to 1$
}\label{L1}

If we let $k\to 1$ in \eqref{RR}, \eqref{zz} we find that bi-cyclide coordinates approach spherical coordinates.
However, by changing the limiting process we show that bi-cyclide coordinates can also approach prolate spheroidal coordinates as $k\to 1$.

Prolate spheroidal coordinates \cite[p.~28]{MoonSpencer} are given by
\[ x=\sinh \sigma\sin\theta\cos\phi,\quad y=\sinh \sigma \sin\theta \sin\phi,\quad
z=\cosh \sigma\cos \theta,\]
where $\sigma\in(0,\infty)$, $\theta\in(0,\pi)$, $\phi\in(-\pi,\pi]$.
According to \cite[\S 245]{Hob} we have the expansion
\begin{eqnarray*}
&& \hspace{-2.5cm}\frac{1}{\|\r-\r^\ast\|}= \sum_{\ell=0}^\infty (2\ell+1)
\sum_{m=-\ell}^\ell (-1)^m\left[\frac{(\ell-m)!}{(\ell+m)!}\right]^2\\
&& \times \P_\ell^m(\cos\theta)\P_\ell^m(\cos\theta^\ast)P_\ell^m(\cosh \sigma)Q_\ell^m(\cosh \sigma^\ast)\expe^{im(\phi-\phi^\ast)},
\end{eqnarray*}
where $\P_\ell^m$ denotes the Ferrers function of the first kind, $P_\ell^m$, $Q_\ell^m$ are associated Legendre functions of the first and second kind, respectively,
$(\sigma,\theta,\phi)$, $(\sigma^\ast,\theta^\ast,\phi^\ast)$ are prolate spheroidal coordinates
of $\r$, $\r^\ast$, respectively, and it is assumed that $\sigma<\sigma^\ast$.
We may write the expansion in the equivalent form
\[ \frac{1}{\|\r-\r^\ast\|}=\sum_{m\in\Z}
 \expe^{im(\phi-\phi^\ast)}\sum_{n=0}^\infty B_{m,n}(\sigma,\sigma^\ast,\theta,\theta^\ast),\]
where $B_{-m,n}=B_{m,n}$, and for $m,n\in\N_0$,
\begin{eqnarray*}
&&\hspace{-2.2cm} B_{m,n}=(-1)^m(2m+2n+1)\left[\frac{n!}{(2m+n)!}\right]^2 \\
 &&\times \P_{m+n}^{m}(\cos\theta)\P_{m+n}^{m}(\cos\theta^\ast)P_{m+n}^{m}(\cosh \sigma)Q_{m+n}^{m}(\cosh \sigma^\ast).
\end{eqnarray*}

We modify bi-cyclide coordinates by setting $\sigma=s+K$, $X=\frac{2}{k'} x$, $Y=\frac{2}{k'}y$,
$Z=\frac{2}{k'}z$. The bi-cyclide coordinates $(\sigma,t,\phi)$ of $(x,y,z)$ and $(X,Y,Z)$ are the same.
If we let $k\to 1$, we obtain
\[ \lim_{k\to 1} X=\sinh \sigma\cos t\cos\phi,\, \lim_{k\to 1} Y=\sinh \sigma\cos t\sin\phi,
\, \lim_{k\to 1} Z=\cosh \sigma\sin t ,\]
so we approach prolate spheroidal coordinates with $\theta=\frac12\pi -t$.

Let us write the expansion \eqref{expansion2a} in the form
\begin{equation}\label{expansionA2}
 \frac{1}{\|\r-\r^\ast\|}=\sum_{m\in\Z} \expe^{im(\phi-\phi^\ast)}\sum_{n=0}^\infty A_{m,n}(\sigma,\sigma^\ast,t,t^\ast,k),
 \end{equation}
where $A_{-m,n}=A_{m,n}$, and, for $m,n\in\N_0$,
\begin{eqnarray*}
 A_{m,n}&=&(X^2+Y^2)^{-1/4}({X^\ast}^2+{Y^\ast}^2)^{-1/4}\\
&&\times \frac{2}{w_m^n}W^n_{m-\frac12}(t,k')W_{m-\frac12}^n(t^\ast,k')
W_{m-\frac12}^n (K'-i\sigma,k')W_{m-\frac12}^n(K'-i(2K-\sigma^\ast)),k').
\end{eqnarray*}
In \eqref{expansionA2} we take $\r=(X,Y,Z)$, $\r^\ast=(X^\ast,Y^\ast,Z^\ast)$.

\begin{lemma}\label{asylemma}
Let $\sigma_0>0$, $\nu\ge -\frac12$, $n\in\N_0$. Then we have
\[ \frac{W_\nu^n(K'-i(2K-\sigma),k')}{W_\nu^n(K'-i(2K-\sigma_0),k')}\to
\frac{(\sinh \sigma)^{1/2}Q_{n+\nu+\frac12}^{\nu+\frac12}(\cosh \sigma)}{(\sinh \sigma_0)^{1/2}Q_{n+\nu+\frac12}^{\nu+\frac12}(\cosh \sigma_0)}
\]
as $k\to 1$ locally uniformly for $\Re \sigma>0$.
\end{lemma}
\begin{proof}
The function $w(\sigma)=W_\nu^n(K'-i(2K-\sigma),k')$, $0<\sigma<2K$,  satisfies the differential equation
\begin{equation}\label{ode5}
 w''+\left(\nu(\nu+1)-\Lambda_\nu^n(k')-\nu(\nu+1)\frac{1}{\sn^2(\sigma,k)}\right)w=0 .
\end{equation}
By \cite[Lemma 2.3]{BiCohlVolkmerB}, $\Lambda_\nu^k(k')\to (n+\nu+1)^2$ as $k\to 1$.
The differential equation \eqref{ode5} appeared in the proof of \cite[Theorem~7.2]{BiCohlVolkmerA2} with $k'$ in place of $k$.
The sequence $\Lambda_\nu^n(k')$ was replaced by another sequence that converged to $n^2$.
We can now follow the proof of \cite[Theorem~7.2]{BiCohlVolkmerA2} to complete the proof the lemma.
\end{proof}

The main result of this section follows.

\begin{thm}\label{limitk1}
Let $m\in\Z$, $n\in\N_0$, $\sigma,\sigma^\ast\in(0,\infty)$, $t,t^\ast\moro{\in}(-\frac12\pi,\frac12\pi)$ and set $\theta=\frac12\pi-t$, $\theta^\ast=\frac12\pi-t^\ast$. Then
\[ A_{m,n}(\sigma,\sigma^\ast,t,t^\ast,k)\to B_{m,n}(\sigma,\sigma^\ast,\theta,\theta^\ast)\quad \text{as $k\to 1$} .\]
\end{thm}
\begin{proof}
It is sufficient to consider $m\ge 0$. All limits in this proof are taken as $k\to 1$.
We first note that
\begin{equation}\label{limit4}
(X^2+Y^2)^{-1/4}({X^\ast}^2+{Y^\ast}^2)^{-1/4}\to (\sinh \sigma \sin\theta)^{-1/2} (\sinh \sigma^\ast \sin \theta^\ast)^{-1/2}.
\end{equation}
Using \cite[Corollary~4.4]{BiCohlVolkmerB}, we have
\begin{eqnarray}\label{limit5}
&&\hspace{-1.0cm} 2W^n_{m-\frac12}(t,k')W_{m-\frac12}^n(t^\ast,k')\nonumber\\
&&  \to (2m+2n+1)\frac{n!}{(2m+n)!}(\sin \theta)^{1/2}
\P_{m+n}^n(\cos \theta)(\sin \theta^\ast)^{1/2} \P_{m+n}^m(\cos \theta^\ast) .
\end{eqnarray}
We define two functions
\begin{eqnarray*}
&&\hspace{-7.0cm}f(u,k):= W_{m-\frac12}^n(u,k'),\\
&&\hspace{-7.0cm}g(u^\ast,k):= W_{m-\frac12}^n(K'-i(2K-u^\ast),k') .
\end{eqnarray*}
These functions are well-defined for $u,u^\ast\in(0,2K)$.
Then we consider the expression
\begin{equation}\label{expr2}
h(u,u^\ast,k)=\frac{f(u,k)g(u^\ast,k)}{[g,f]},
\end{equation}
where $w_m^n=[g,f]$ denotes the Wronskian of $g(\cdot,k)$ and $f(\cdot,k)$.
We notice that \eqref{expr2} remains unchanged when we multiply $f$ and/or $g$ by real or complex constants.
Therefore, using \cite[Corollary~4.4]{BiCohlVolkmerB} and Lemma \ref{asylemma} one obtains
\[
h(\sigma,\sigma^\ast,k)\to \frac{F(\sigma)G(\sigma^\ast)}{[G,F]},
\]
where
\[ F(\sigma)=(\sinh \sigma)^{1/2}P_{m+n}^m(\cosh \sigma),\quad G(\sigma^\ast)=(\sinh \sigma^\ast)^{1/2} Q_{m+n}^m(\cosh \sigma^\ast). \]
The known Wronskian \cite[(14.2.10)]{NIST:DLMF}
\[ [P_\nu^\mu(x),Q_\nu^\mu(x)]=\expe^{i\mu \pi} \frac{\Gamma(\nu+\mu+1)}{\Gamma(\nu-\mu+1)}\frac{1}{1-x^2}, \]
implies
\[ [G,F]=(-1)^m \frac{(2m+n)!}{n!} .\]
Thus we have shown
\begin{eqnarray}\label{limit6}
&&\hspace{-1cm}\frac{1}{w_m^n}W_{m-\frac12}^n (K'-i\sigma,k')W_{m-\frac12}^n(K'-i(2K-\sigma^\ast),k')\nonumber\\
&&\hspace{0.2cm}\to (-1)^m \frac{n!}{(2m+n)!} (\sinh \sigma)^{1/2}P_{m+n}^n(\cosh \sigma)(\sinh \sigma^\ast)^{1/2} Q_{m+n}^n(\cosh \sigma^\ast).
\end{eqnarray}
After multiplying out \eqref{limit4}, \eqref{limit5}, \eqref{limit6} and minor simplification, we obtain the desired statement.
\end{proof}

\appendix

\section{The bi-cyclide coordinates of Moon and Spencer}

Moon and Spencer \cite[p.~124]{MoonSpencer} define bi-cyclide coordinates $\mu,\nu,\phi$ by
\begin{eqnarray*}\label{MS}
&&\hspace{-6.3cm}x=  \frac{a}{\Lambda} \cn(\mu,\kappa)\dn(\mu,\kappa)\sn(\nu,\kappa')\cn(\nu,\kappa')\cos\phi,\\
&&\hspace{-6.3cm}y=\frac{a}{\Lambda} \cn(\mu,\kappa)\dn(\mu,\kappa)\sn(\nu,\kappa')\cn(\nu,\kappa')\sin\phi,\\
&&\hspace{-6.3cm}z= \frac{a}{\Lambda} \sn(\mu,\kappa)\dn(\nu,\kappa'),
\end{eqnarray*}
where
\[ \Lambda=1-\dn^2(\mu,\kappa)\sn^2(\nu,\kappa'), \]
$a$ is a positive constant, and $\kappa\in(0,1)$, $\kappa'=(1-\kappa^2)^{1/2}$.
Setting $R=(x^2+y^2)^{1/2}$ and using the addition theorem for the Jacobi function $\sn$ \cite[(22.8.1)]{NIST:DLMF},
we can write these coordinates in the complex form
\begin{equation}\label{bi1}
z+iR= a\sn(\mu+i\nu,\kappa) .
\end{equation}
The function $v=\sn(u,\kappa)$ maps the rectangle
\[
 -K(\kappa)<\Re u<K(\kappa), \quad 0<\Im u<K'(\kappa),
 \]
conformally to the half-plane $\Im v>0$.
Therefore, we choose
\begin{equation}\label{range1}
-K(\kappa)<\mu<K(\kappa),\quad 0<\nu<K'(\kappa).
\end{equation}

Wangerin \cite{Wangerin1875} introduced coordinates $\mu,\nu$ in the $(R,z)$-plane by setting $z+i R=af(\mu+i\nu)$
for $f=\cn$, $f=\sn$ and $f=\dn$. Actually, he considers only $f=\cn$ and $f=\dn$ because the coordinates
generated by $\sn$ and $\dn$ are essentially the same. This follows from the identity
\[ \kappa\sn(u,\kappa)=\dn(K'(\kappa)+i K(\kappa)-i u,\kappa') .\]

In this paper we used bi-cyclide coordinates $s\in(-K,K)$, $t\in(-K',K')$ defined by \eqref{RR}, \eqref{zz}.
They can be written in complex form as
\begin{equation}\label{bi2}
 z+iR= i(\ssc(s-it,k)+\nc(s-it,k)) .
\end{equation}

The connection between $\mu,\nu$ and $s,t$ is given by the following theorem.

\begin{thm}\label{MoonSpencer}
Take $a=\kappa^{1/2}$ and $\kappa=\frac{1-k}{1+k}$. Then the coordinates $\mu,\nu$ and $s,t$ of a point $(R,z)$ with $R>0$ are connected by
\[ t=(1+\kappa)\mu,\quad s+K(k)=(1+\kappa) \nu .\]
\end{thm}
\begin{proof}
The modulus $\kappa$ is the descending Landen transformation of $k'$
so \cite[(19.8.12)]{NIST:DLMF} gives
\[ K'(k)=(1+\kappa)K(\kappa),\quad 2K(k)=(1+\kappa)K'(\kappa).\]
It follows that $\tilde s:=(1+\kappa)\nu-K(k)$, $\tilde t:=(1+\kappa)\mu$ satisfy
$-K(k)<\tilde s<K(k)$, $-K'(k)<\tilde t<K'(k)$.
Therefore, using \eqref{bi1}, \eqref{bi2} and setting $u=(1+\kappa)(\mu+i\nu)$, the theorem will follow from
the identity
\begin{equation}\label{toshow}
\kappa^{1/2}\sn\left(\frac{u}{1+\kappa},\kappa\right)=i\left(\nc(iu+K(k),k)-\ssc(iu+K(k),k)\right) .
\end{equation}
To prove \eqref{toshow} we note that
\begin{eqnarray*}
 &&\hspace{-2.0cm}i(\nc(iu+K(k),k)-\ssc(iu+K(k),k))\\
 &&\hspace{1cm}=\frac{i}{k'}\left(\cs(iu,k)-\ds(iu,k)\right)=\frac{1}{k'}\left(\ns(u,k')-\ds(u,k')\right) .
\end{eqnarray*}
Now \eqref{toshow} follows from \cite[(22.7.2), (22.7.4)]{NIST:DLMF}.
\end{proof}

\bibliographystyle{plain}

\begin{thebibliography}{}

\end{thebibliography}


\begin{thebibliography}{10}

\bibitem{BiCohlVolkmerA}
L.~{Bi}, H.~S. {Cohl}, and H.~{Volkmer}.
\newblock {Expansion for a fundamental solution of Laplace's equation in
  flat-ring cyclide coordinates}.
\newblock {\em Symmetry, Integrability and Geometry:~Methods and Applications
  (SIGMA)}, 18:Paper 041, 31, 2022.

\bibitem{BiCohlVolkmerA2}
L.~{Bi}, H.~S. {Cohl}, and H.~{Volkmer}.
\newblock {Expansion for a fundamental solution of Laplace's equation in
  flat-ring cyclide coordinates}.
\newblock {\em {\tt to appear in}:~{\it Symmetry, Integrability and
  Geometry:~Methods and Applications (SIGMA)}}, 2022.

\bibitem{BiCohlVolkmerB}
L.~{Bi}, H.~S. {Cohl}, and H.~{Volkmer}.
\newblock {Peanut harmonic expansion for a fundamental solution of Laplace's
  equation in flat-ring coordinates}.
\newblock {\em Analysis Mathematica}, 48, 2022.

\bibitem{CT}
H.~S. {Cohl} and J.~E. {Tohline}.
\newblock {A Compact Cylindrical Green's Function Expansion for the Solution of
  Potential Problems}.
\newblock {\em The Astrophysical Journal}, 527:86--101, 1999.

\bibitem{NIST:DLMF}
{\it NIST Digital Library of Mathematical Functions}.
\newblock \href{https://dlmf.nist.gov/}{\bf\tt\normalsize
  https://dlmf.nist.gov/}, Release 1.1.8 of 2022-12-15.
\newblock F.~W.~J. Olver, A.~B. {Olde Daalhuis}, D.~W. Lozier, B.~I. Schneider,
  R.~F. Boisvert, C.~W. Clark, B.~R. Miller, B.~V. Saunders, H.~S. Cohl, and
  M.~A. McClain, eds.

\bibitem{ErdelyiHTFIII}
A.~Erd{\'e}lyi, W.~Magnus, F.~Oberhettinger, and F.~G. Tricomi.
\newblock {\em Higher Transcendental Functions. {V}ol. {III}}.
\newblock Robert E. Krieger Publishing Co. Inc., Melbourne, Fla., 1981.

\bibitem{Hob}
E.~W. Hobson.
\newblock {\em The theory of spherical and ellipsoidal harmonics}.
\newblock Chelsea Publishing Company, New York, 1955.

\bibitem{Kellogg}
O.~D. Kellogg.
\newblock {\em Foundations of potential theory}.
\newblock Reprint from the first edition of 1929. Die Grundlehren der
  Mathematischen Wissenschaften, Band 31. Springer-Verlag, Berlin, 1967.

\bibitem{Miller}
W.~Miller, Jr.
\newblock {\em Symmetry and separation of variables}.
\newblock Addison-Wesley Publishing Co., Reading, Mass.-London-Amsterdam, 1977.
\newblock With a foreword by Richard Askey, Encyclopedia of Mathematics and its
  Applications, Vol. 4.

\bibitem{MoonSpencer}
P.~Moon and D.~E. Spencer.
\newblock {\em Field theory handbook, including coordinate systems,
  differential equations and their solutions}.
\newblock Springer-Verlag, Berlin, 1961.

\bibitem{MorseFeshbach}
P.~M. Morse and H.~Feshbach.
\newblock {\em Methods of theoretical physics. 2 volumes}.
\newblock McGraw-Hill Book Co., Inc., New York, 1953.

\bibitem{Volkmer84}
H.~Volkmer.
\newblock Integral representations for products of {L}am\'e functions by use of
  fundamental solutions.
\newblock {\em SIAM Journal on Mathematical Analysis}, 15(3):559--569, 1984.

\bibitem{Volkmer2018}
H.~Volkmer.
\newblock Eigenvalue problems for {L}am\'{e}'s differential equation.
\newblock {\em Symmetry, Integrability and Geometry: Methods and Applications
  (SIGMA)}, 14:131, 21 pages, 2018.

\bibitem{Wangerin1875}
A.~{Wangerin}.
\newblock {Reduction der Potentialgleichung f\"ur gewisse Rotationsk\"orper auf
  eine gew\"ohnliche Differentialgleichung.}
\newblock {Preisschr. der Jabl. Ges. Leipzig, Hirzel}, 1875.

\end{thebibliography}

\def\cprime{$'$} \def\dbar{\leavevmode\hbox to 0pt{\hskip.2ex \accent"16\hss}d}

\end{document}